\documentclass[12pt,centertags,oneside]{article}
\usepackage{amsmath,amstext,amsthm,amscd,typearea}
\usepackage{amssymb}
\usepackage{a4wide}
\usepackage[mathscr]{eucal}
\usepackage{mathrsfs}
\usepackage{typearea}
\usepackage{charter}
\usepackage{pdfsync}
\usepackage{url}
\usepackage[T1]{fontenc}

\usepackage{hyperref}

\usepackage{xcolor}

\usepackage[a4paper,width=16.2cm,top=3cm,bottom=3cm]{geometry}

\numberwithin{equation}{section}

\newtheorem{theorem}{Theorem}[section]

\newtheorem{proposition}[theorem]{Proposition}
\newtheorem{corollary}[theorem]{Corollary}
\newtheorem{lemma}[theorem]{Lemma}

\newtheorem{example}[theorem]{Example}

\newcommand{\supp}{{\rm Supp}}

\newcommand{\Leb}{\mathop{\mathrm{Leb}}\nolimits}

\newcommand{\dist}{\mathop{\mathrm{dist}}\nolimits}

\newcommand{\ddc}{dd^c}
\newcommand{\dc}{d^c}
\def\d{\operatorname{d}}

\newcommand{\C}{\mathbb{C}}
\newcommand{\D}{\mathbb{D}}
\newcommand{\N}{\mathbb{N}}

\title{\bf Continuity of functions in complex Sobolev spaces}
\providecommand{\keywords}[1]{\textbf{\textit{Keywords:}} #1}
\providecommand{\subject}[1]{\textbf{\textit{Mathematics Subject Classification 2010:}} #1}

\author{Duc-Viet Vu}
\newcommand{\Addresses}{{
		\bigskip
		\footnotesize
		\textsc{Duc-Viet Vu, University of Cologne, Division of Mathematics, Department of Mathematics and Computer Science, Weyertal 86-90, 50931, K\"oln.}
		\noindent
		\par\nopagebreak
		\noindent
		\textit{E-mail address}: \texttt{dvu@uni-koeln.de}
}}

\date{\today}

\begin{document}

\maketitle

\begin{abstract} We study the continuity regularity of functions in the complex Sobolev spaces. As applications, we obtain Hermitian generalizations of a recent result due Guedj-Guenancia-Zeriahi on the diameters of K\"ahler metrics. 
\end{abstract}
\noindent
\keywords {Monge-Amp\`ere equation}, {diameter}, {complex Sobolev space}, {closed positive current}.
\\

\noindent
\subject{32U15},  {32Q15}.

%\tableofcontents

%%%%%%%%%%%%%%%%%%%%%%%%%%%%%
%%%%%%%%%%%%%%%%%%%%%%%%%%%%%%%%%%

\section{Introduction}

Let $U$ be a bounded open set in $\C^n$ endowed with the standard Euclidean form $\omega_{\C^n}:= \frac{1}{2}dd^c |z|^2$ (recall $d=\partial + \bar{\partial}$ and $d^c:=\frac{i}{2 \pi} (\bar \partial -\partial)$). Let $W^*(U)$ be the subset of $W^{1,2}(U)$ consisting of $u$ satisfying that  there exists a positive closed current $T$ of bidegree $(1,1)$ and of finite mass (i.e., $\int_U T \wedge \omega_{\C^n}^{n-1} < \infty$) on $U$ such that  
\begin{equation}\label{bound_by_current}
d \varphi \wedge d^c \varphi \leq T.
\end{equation}
The space $W^*(U)$ is called complex Sobolev space.  When $k=1$, then $W^*(U)=W^{1,2}(U)$ since $d \varphi \wedge d^c \varphi$  is already a measure. The space $W^*(U)$ was introduced by Dinh and Sibony in \cite{DS_decay} and developed in \cite{Vigny} in the study of complex dynamics. We refer to \cite{DLW,DLW2,Vigny_expo-decay-birational,Vu_nonkahler_topo_degree} for more applications in dynamics; see also \cite{DinhMarinescuVu,DKC_Holder-Sobolev} for recent applications in Monge-Amp\`ere equations.

The space $W^*(U)$ is actually a Banach space endowed with the norm
$$ \|\varphi\|^2_*= \int_U |\varphi|^2 \omega_{\C^n}^{n} + \inf\big\{\int_U T \wedge \omega_{\C^n}^{n-1} \big\}$$    
where the infimum is taken over all the positive closed current of bidegree $(1,1)$ satisfying \eqref{bound_by_current}; see \cite{Vigny}.  Let $\gamma>0$ be a constant.   A function $f$ is said to be $\log^\gamma\log$-continuous on a bounded set $K \subset \C^n$ if there is a constant $B>0$ such that  for every $x,y \in K$ we have 
$$|f(x)- f(y)| \le \frac{B}{\max\{\log^\gamma |\log |x-y||,1\}}\cdot$$
We let $\|f\|_{\log^\gamma \log(K) }$ to be the sum of $\|f\|_{L^\infty(K)}$ and  the  minimum of all such constants $B$. 

Let $M \ge 1$ and $\gamma_0>0$ be  constants. Let $\mathcal{A}_{M,\gamma_0}= \mathcal{A}_{M,\gamma_0}(U)$ be the subset of $W^*(U)$ consisting of $u \in W^*(U)$ such that there is a psh function $\psi$ which is $\log^{1+\gamma_0} \log$- continuous function with $\|\psi\|_{\log^{1+\gamma_0} \log (U)} \le M$   satisfying $d u \wedge \dc u \le  \ddc \psi$. Let $\gamma_1 \in (0,1]$, we define $\mathcal{A}'_{M,\gamma_1}$ similarly as $\mathcal{A}_{M,\gamma_0}$ but with the H\"older norm $\|\psi\|_{\mathcal{C}^{\gamma_1}(U)}$ instead of $\|\psi\|_{\log^{1+\gamma_0} \log (U)}$. Note that we don't require that the $*$-norm of $u \in \mathcal{A}_{M,\gamma_0}$ (or $\mathcal{A}'_{M,\gamma_1}$) is uniformly bounded.  Here is the main result of this paper.

\begin{theorem} \label{th-bichanLinfinity}  Let $K$ be a compact subset in $U$. Then, for every constant $\gamma \in (0, \gamma_0)$,  there exists a constant $C_{M,K, \gamma}$ such that 
$$|u(x)- u(y)| \le \frac{C_{M,K,\gamma}}{\max\{\log^{\gamma/2} |\log |x-y||,1\}},$$
for every $u \in\mathcal{A}_{M,\gamma_0}$ and $x,y \in K$. Similarly, for every $\gamma'_1 \in (0, \gamma_1/2)$, and 
there exists a constant $C_{M,K,\gamma_1'}$ satisfying
$$|u(x)-u(y)| \le C_{M,K,\gamma_1'}|x-y|^{\gamma'_1}$$
for every $u \in \mathcal{A}'_{M,\gamma_1}$ and $x,y \in K$.
\end{theorem}

%We note that if the $\log^{1+\gamma_0} \log$ regularity is replaced by H\"older continuity then we also obtain a corresponding statement. 
%WRITE A VERSION FOR HOLDER CONTINUITY. HERMITIAN SETTING IS NEW, BECAUSE THE PROOF IN [YANG LI] USES THE KAEHLER HYPOTHESIS. CHECK [ BIN GUO, MODULUS], IT SEEMS THAT THE KAHLER HYPOTHESIS IS NOT NECESSARY.

A compact (K\"ahler) version for a more specific class of $u$ was proved recently in \cite[Theorems 3.4 and 4.4]{GGZ-logcontiu} for $\log^{1+\gamma_0}\log$-continuity and \cite[Theorem 4.1]{YangLi} for H\"older continuity (and $\gamma'_1= \gamma_1/2$). To go into details, let us consider a compact K\"ahler manifold $X$. Let $\omega$ be a smooth K\"ahler form on $X$ and $d_\omega$ be the Riemannian distance induced by $\omega$. Let $x_0 \in X$ and $f(x):=d_\omega(x_0,x)$ for $x  \in X$.
The proofs of \cite[Theorems 3.4 and 4.4]{GGZ-logcontiu} use extensively the fact that   $df \wedge \dc f \le \omega$. Such an inequality was used several times in other previous papers such as \cite{Guo-Phong-Song-Sturm,YangLi}.  

The above estimate for $f$ also holds for singular metrics as observed in \cite{GGZ-logcontiu}. To be precise, let $T$ be a closed positive $(1,1)$-current on $X$ such that $T$ is a smooth K\"ahler form on the complement of a proper analytic subset $V$ in $X$. Let $d_T$ be the Riemannian distance induced by $T$ on $X \backslash V$. Let $x_0 \in X \backslash V$ and let $f(x):= d_T(x_0, x)$ for $x \in X \backslash V$. Then we have $df \wedge \dc f \le T$ on $X \backslash V$ (see Lemma \ref{le-dfwedgedcf} below). This, in particular, implies that $df \in L^2(X)$. Hence by \cite[Proposition 3.1]{DS_decay}, one sees that $f \in W^{1,2}(X)$ and the inequality $df \wedge \dc f \le T$ holds as currents on $X$. In this context, it was proved in \cite{GGZ-logcontiu} that if $T$ has  $\log^{1+\gamma_0} \log$ continuous potentials and $\gamma \in (0,\gamma_0)$ is a constant, then $f$ is $\log^{\gamma/2} \log$-continuous with uniform constants.  The corresponding statement for H\"older regularity was established in \cite{YangLi} for $\gamma'_1=\gamma_1/2$ (see also \cite{Guo-Phong-Tong-Wang}).  
  We state now more general version of these results for smooth holomorphic families of Hermitian manifolds.

\begin{corollary}\label{cor-family} Let $\pi: \mathcal{X} \to Y$ be a proper holomorphic submersion, where $\mathcal{X}$ and $Y$ are complex manifolds. Let $\omega$ be a Hermitian metric on $\mathcal{X}$. Let $\psi$ be an $\omega$-psh function which is  $\log^{1+\gamma_0}\log$-continuous on $\mathcal{X}$ such that $\psi$ is smooth outside some proper analytic subset  $V \subset \mathcal{X}$. Let $T:= \ddc \psi + \omega$, and for every $y \in Y$, put $X_y:= \pi^{-1}(y)$. Let $d_{T,y}$ be the Riemannian distance induced by $T$ on $X_y \backslash V$ (if $X_y \not \subset V$).  Then for every compact $K \subset Y$ and for every $\gamma \in (0, \gamma_0)$, there exists a  constant $C_\gamma>0$ such that for every $y \in K$, if $X_y \not \subset V$, then  we have
$$d_{T,y}(x_1,x_2) \le \frac{C_\gamma}{ \max\{\log^{\gamma/2}|\log |x_1-x_2|,1\}}$$  
for every $x_1,x_2 \in X_y \backslash V$, where $|x_1-x_2|$ denotes the distance between $x_1,x_2$ induced by $\omega$ on $\mathcal{X}$.  In particular the diameters of $(X_y\backslash V,d_{T,y})$ is uniformly bounded for $y \in K$ (and bounded by a constant depending only on $\|\psi\|_{\log^{1+\gamma_0} \log(\mathcal{X})}$ and $\omega, \pi$). 

Moreover if $\psi$ is H\"older continuous with H\"older exponent $\gamma_1 \in (0, 1]$, then for every $\gamma'_1 \in (0,\gamma_1/2)$, there holds 
$$d_{T,y}(x_1,x_2) \le C_{\gamma'_1} |x_1-x_2|^{\gamma'_1}$$  
for every $x_1,x_2 \in X_y \backslash V$ and $y \in K$.
\end{corollary}

We refer to recent works \cite{Guo-Phong-Song-Sturm,Guo-Phong-Song-Sturm2} by Guo-Phong-Song-Sturm for  very strong general results on uniform diameter bound for K\"ahler metrics. The difference of our results (as well as those in \cite{GGZ-logcontiu,YangLi})  to \cite{Guo-Phong-Song-Sturm,Guo-Phong-Song-Sturm2} is that we don't require the Monge-Amp\`ere measure of the metric has a small vanishing locus (with respect to a fixed volume form). More importantly, Corollary \ref{cor-family} is true in the Hermitian setting (as far as we can see, the proof of \cite[Theorem 4.4]{GGZ-logcontiu} does not extend immediately to the Hermitian setting). %, whereas this is the case for the proof of \cite[Theorem 4.1]{YangLi}). %This is a new point compared to \cite{GGZ-logcontiu,Guo-Phong-Song-Sturm,Guo-Phong-Song-Sturm2,YangLi}. 

We will prove Theorem \ref{th-bichanLinfinity} as follows.  By using slicing of currents, we reduce  the problem to the case of dimension 1. The desired continuity is  now obtained by repeating concrete computations in \cite{GGZ-logcontiu}. % Besides one must handle also some more terms due to the non-compactness of $U$. 
\\

\noindent
\textbf{Acknowledgments.} We thank Ngoc Cuong Nguyen, Henri Guenancia, and Gabriel Vigny for fruitful discussions. We also want to express our great gratitude to the anonymous referee for his/her careful reading and suggestions.  The research of the author is partially supported by the Deutsche Forschungsgemeinschaft (DFG, German Research Foundation)-Projektnummer 500055552 and by the ANR-DFG grant QuaSiDy, grant no ANR-21-CE40-0016.

\section{Proof of main results}

In this section, we first prove Theorem \ref{th-bichanLinfinity}.   We start with the following known observation.

\begin{lemma}\label{le-dfwedgedcf} Let $(X,\omega)$ be a Hermitian manifold. Let $d_\omega$ be the Riemannian distance induced by $\omega$.  Let $x_0 \in X$ and $f(x):= d_\omega(x_0,x)$ for $x \in X$. Then we have 
$$df \wedge \dc f \le \omega$$
as currents on $X$. 
\end{lemma}

\proof We include a proof for readers' convenience. The desired inequality is pointwise. Hence it suffices to work in a  local chart $(U, z_1,\ldots,z_n)$ near a point $x \in X$. Observe that $f$ is Lipschitz with the Lipschitz norm bounded by 1 because of the triangle inequality. Thus $f$ is differentiable almost everywhere and $|\nabla f|_\omega \le 1$. 
%Recall that $|\nabla f|_\omega =1$.  
Diagonalizing $\omega$ at $x$, we get $$\omega= i \sum_{j=1}^n d z_j \wedge d \bar z_j$$
at $x$.  Hence 
$$1= |\nabla f(x)|_\omega^2 = 2\sum_{j=1}^n |\partial_{z_j} f(x)|^2.$$
It follows that the trace of the form  $\eta:= i \partial f \wedge \bar \partial f$ with respect to $\omega$ at $x$ is equal to $1/2$. Using a unitary transform at $x$ to diagonalize $\eta$ at $x$, we see that 
$$\eta= \sum_{j=1}^n a_j d z_j \wedge d \bar z_j,$$
for $a_j \ge 0$ and $\sum_{j=1}^n a_j=1/2$. Consequently we obtain $\eta \le \omega$ at $x$. This combined with the fact that 
$$df \wedge \dc f = \frac{i}{\pi} \partial f \wedge \bar \partial f$$
gives the desired inequality.  This finishes the proof. 
\endproof

Next we recall a fact about slicing of currents. Let $U$ and $V$ be bounded open subsets of $\C^{m_1}$ and $\C^{m_2}$ respectively. Let $\pi_U: U \times V \to U$ and $\pi_V: U \times V \to V$ be the natural projections. %Observe that  if $R$ is a form with $L^1_{loc}$ coefficients (which is not necessarily closed or positive), we can always define the restriction $R_z$ of $R$ to the fiber $\pi_V^{-1}(z)$ for almost every  $z\in V$ (with respect to  the Lebesgue measure on $V$).  

Consider now a closed positive $(1,1)$-current $R$ on $U \times V$.  Write $R= \ddc w$ locally, where $w$ is a psh  function. Let $A$ be the set of $z$ so that  $w(\cdot, z) \equiv -\infty$. Observe that $A$ is pluripolar. To see it, let $K$ be a compact subset in $U$ having a non-empty interior and let $\Leb_U$ be the Lebesgue measure on $U$. Then $\int_{x \in K} w(x,z) d \Leb_U$ is a psh function on $V$ and this function is equal to $-\infty$ on $A$. This implies that $A$ is pluripolar. 

For $z\in V \backslash A$, we define the slice $R_z$ of  $R$ on $\pi_V^{-1}(z)$ to be $\ddc \big(w(\cdot, z)\big)$ which is a closed positive $(1,1)$-current on $\pi_V^{-1}(z)$. One can see that the definition  of the slice $R_z$ is independent of the choice of a local potential $w$ of $R$. 
 
%Let $\chi$ be a nonnegative smooth radial function with compact support on $\C^{m_2}$ such that $\int_{\C^{m_2}} \chi d \Leb =1$ and for every constant $\varepsilon >0$, we put $\chi_{\varepsilon}(z):= \varepsilon^{-2 m_2} \chi(\varepsilon^{-1}z)$. 

\begin{lemma} \label{le-chondiemtam2} (\cite[Lemma 3.4]{DinhMarinescuVu}) Let $u$ be a locally integrable function in $U \times V$ such that $\partial u \in L^2_{loc}(U \times V)$.  Let $T$ be   a closed positive $(1,1)$-current  on $U \times V$ such that $i \partial u \wedge \overline \partial u \le T.$   Then, for almost every $z \in V$, we have that $\partial (u |_{U \times \{z\}}) \in L^2_{loc}(U)$ and  
\begin{align}\label{ine-sliceuPsix'W12}
i \partial (u |_{U \times \{z\}}) \wedge \overline \partial (u |_{U \times \{z\}}) \le T |_{U \times \{z\}}.
\end{align}
\end{lemma}

We note that if $T$ has $\log^\gamma\log$-continuous potentials, then so does $T|_{U\times \{z\}}$. 

\begin{proof}[Proof of Theorem \ref{th-bichanLinfinity}]
We first prove the case where $\psi$ is $\log^{1+\gamma_0} \log$-continuous. Let $u_\epsilon$, $\psi_\epsilon$ be the standard convolutions of $u,\psi$ respectively. We have $u_\epsilon \to u$ in $L^2$ (see \cite{Vigny-Vu-Lebesgue,DinhMarinescuVu} for a much better property that $u_\epsilon$ converges pointwise outside a pluripolar set and in capacity to $u$, but we don't need this in the proof), and 
$$d u_\epsilon \wedge \dc u_\epsilon \le \ddc \psi_\epsilon,$$
(see \cite[Lemma 5]{Vigny} or \cite[Lemma 3.3]{DinhMarinescuVu}). Hence $\psi_\epsilon$ is of the following form:
$$\psi_\epsilon(x)= \int_U \psi(x-z) \chi_\epsilon(z) \omega_{\C^n}^n,$$
for $x \in U$ with $\dist(x, \partial U)> \epsilon$, where $\chi_\epsilon$ is smooth and supported on the ball of radius $\epsilon$ centred at $0$ in $\C^n$ such that $\int_U \chi_\epsilon \omega_{\C^n}^n =1$. It follows that 
$$|\psi_\epsilon(x) - \psi_\epsilon(y)| \le \int_U |\psi(x-z)- \psi(y-z)| \chi_\epsilon(z) \omega_{\C^n}^n \le \frac{B}{\max\{\log^\gamma |\log |x-y|,1\}}$$
for some constant $B>0$ independent of $x,y, \epsilon$ with $\dist(x, \partial U)> \epsilon$, $\dist(y, \partial U)> \epsilon$. In other words, the norm $\|\psi_\epsilon\|_{\log^{1+\gamma_0 \log(U_1)}}$ is bounded uniformly in $\epsilon$.  Hence without loss of generality, we can assume that $u$ and $\psi$ are smooth. 

By Lemma \ref{le-chondiemtam2} and slicing currents by the complex lines passing through a point in $K \subset U$, we see that it suffices to treat the case $n=1$. At this point we follow more or less computations in \cite{GGZ-logcontiu}. The proof in \cite{GGZ-logcontiu} makes use of pluricomplex functions. We notice that a similar idea, using (Riemannian) Green function instead, was used previously in \cite{Guo-Phong-Song-Sturm}.  We recall details  for readers' convenience.

Now as just mentioned above, we assume $n=1$. Since the problem is local, we can assume that $U$ is a bounded connected subset in $\C$ with a smooth boundary, and $u \in \mathcal{A}_{M,\gamma_0}(U')$ for some open set $U'$ containing $\overline U$. Observe now that by Poincar\'e inequality, there is a constant $C>0$ independent of $u$ such that 
$$\int_U \bigg|u- \int_U u \omega_\C \bigg|^2 \omega_\C  \le C \int_U du \wedge \dc u \le C \int_U \ddc \psi  \lesssim M.$$
Hence, by considering $u- \int_U  u \omega_\C$ in place of $u$, we can assume that \begin{align}\label{ine-chuansaocuau}
\|u\|_* \le M,
\end{align}
for $u \in \mathcal{A}_{M,\gamma_0}(U)$. 

Let $U_1,U_2$ be relatively compact open subsets in $U$ such that $K\subset U_1 \Subset U_2$. Let  $\rho$ be a smooth cut-off function on $U$ such that $0 \le \rho \le 1$ and $\rho= 1$ on $U_1$ and $\supp \rho \Subset U_2$.

Let $x,y \in K$. In what follows we use $\lesssim$ or $\gtrsim$ to denotes $\le$ or $\ge$ modulo a positive multiplicative constant independent of $u,x,y,\epsilon$ (below).  Let $g_x(z):= \log |z-x|$ for $z \in U$. We have $\ddc g_x = \delta_x$ the Dirac mass at $x$. Since the problem is local we can assume indeed that $g_x < -10$ for every $x \in K$ (by shrinking $U$ if necessary).

For every constant $\epsilon>0$, let $g_{x, \epsilon}(z):= \log (|z-x|+ \epsilon)$ which decreases to $g_x(z)$ as $\epsilon \to 0$. Put $h_\epsilon:= g_{x, \epsilon}- g_{y,\epsilon}$.  It follows that 
$$u(x)- u(y)= \int_U \rho u \ddc g_x - \int_U \rho u \ddc g_y = \lim_{\epsilon \to 0} \int_U \rho u \ddc h_\epsilon.$$
Let $p(t):= (-t)\big(\log (-t)\big)^{1+\gamma}$ for $t<-10$ and for  some constant $\gamma$ to be determined later. Let 
$$\tilde{p}(z):= p(g_x(z)+ g_y(z)) >0,$$
for $z \in U$. 
By integration by parts, one has 
\begin{align*}
I &:= \int_U \rho u \ddc h_\epsilon = -\int_U \rho d u \wedge \dc h_\epsilon - \int_U u d \rho \wedge \dc h_\epsilon.
\end{align*}

Since $\|h_{\epsilon}\|_{C^1(U_2 \backslash U_1)} \lesssim |x-y|$ and $d \rho$ vanishes on $U_1$, using Cauchy-Schwarz inequality, we infer
$$|I| \lesssim |x-y|\|u\|_{L^1(U_2)}+ \bigg(\int_U \rho \tilde{p} du \wedge \dc u\bigg)^{1/2} \bigg(\int_U \rho \tilde{p}^{-1}d h_\epsilon \wedge \dc h_\epsilon \bigg)^{1/2}.$$
Let 
$$I_1:= \int_U \rho \tilde{p} du \wedge \dc u, \quad I_2:= \int_U \rho \tilde{p}^{-1}d h_\epsilon \wedge \dc h_\epsilon.$$ % denote the first and second integrals in the right-hand side of the last inequality. 

Now we want to estimate $I_2$. In order to do so, we just split $\supp \rho$ into three regions as in \cite{GGZ-logcontiu}. Without loss of generality, we assume that $x=0$. It suffices to consider $y$ close to $x$. Let $\delta:= |x-y| < e^{-100}$.   Let $A_1:= \{|z| \ge 1\}$. Let $A_2:= \{|z| \le 2 \delta \}$ and $A_3:= \{ 2 \delta \le |z| \le 1\}$. One has
$$\tilde{p}^{-1}(z)d h_\epsilon \wedge \dc h_\epsilon(z) \lesssim |x-y|$$
on $A_1$. Hence 
\begin{align}\label{ine-danhgiaI21}
\int_{A_1} \rho \tilde{p}^{-1} d h_\epsilon \wedge \dc h_\epsilon \lesssim |x-y|.
\end{align}
On the other hand
\begin{align}\label{ine-danhgiaI22}
\int_{A_2} \tilde{p}^{-1}(z)d h_\epsilon \wedge \dc h_\epsilon(z) &\lesssim \int_{A_2} \frac{\omega_{\C}}{|z||y-z|\big(-\log |z|- \log|y-z|\big) \log^{1+\gamma} \big|\log |z|+ \log|y-z| \big|}\\
\nonumber
& \lesssim \int_{A_2 \cap \{|z-y| \le |z|\}} \frac{\omega_{\C}}{|y-z|^2 \big|\log|y-z|\big| \log^{1+\gamma} \big|\log|y-z| \big|}\\
\nonumber
& + \int_{A_2 \cap \{|z-y| \ge |z|\}} \frac{\omega_{\C}}{|z|^2 \big|\log|z|\big| \log^{1+\gamma} \big|\log|z| \big|}\\
\nonumber
& \lesssim - \int_{0}^\delta \frac{dr}{r \log r \log^{1+\gamma}(-\log r)}
\end{align}
which is, by the change of variables $t:=  \log (-\log r)$, equal to 
$$\int_{\log (-\log \delta)}^\infty  \frac{dt}{t^{1+\gamma}}= \gamma^{-1} \log^{-\gamma} (-\log |x-y|).$$
Now we consider
$$\int_{A_3} \tilde{p}^{-1}(z)d h_\epsilon \wedge \dc h_\epsilon(z)$$
As in \cite[(3.10)]{GGZ-logcontiu}, direct computations show that 
$$d h_\epsilon \wedge \dc h_\epsilon(z) \lesssim \frac{|x-y|^2}{|x-z|^2|y-z|^2}\bigg(1+ \frac{|x-y|^2}{|z-y|^2}\bigg)^2 \omega_{\C}.$$
It follows that (note $|y-z| \ge |z| - |y|= |z|- \delta \gtrsim  |z|$ on $A_3$)
\begin{align}\label{ine-danhgiaI23}
\int_{A_3} \tilde{p}^{-1}(z)d h_\epsilon \wedge \dc h_\epsilon(z) &\lesssim 
\delta^2 \int_{A_3} \frac{ \omega_{\C}}{|z|^4 \big|\log |z|\big| \log^{1+\gamma}(-\log |z|)}\\
\nonumber
& \lesssim \delta^2 \int_{ 2 \delta  \le r \le 1} \frac{dr}{r^3 \log (-r) \log^{1+\gamma}(-\log r)} 
\end{align}
which is, by the change of variables $t= -\log r$, equal to
$$\delta^2 \int_{- \log 1 \le t \le - \log (2 \delta)} \frac{e^{2t} dt}{t\log^{1+\gamma} t} \lesssim  \frac{1}{ \log^{1+\gamma}(-\log \delta)} \cdot$$
Combining (\ref{ine-danhgiaI23}), (\ref{ine-danhgiaI22}) with (\ref{ine-danhgiaI21}) gives
\begin{align}\label{ine-danhgiaI200}
I_2 \lesssim \frac{1}{ \log^{\gamma}(-\log |x-y|)} \cdot
\end{align}
It remains to bound $I_1$. 
Direct computations give
$$p'(t)= -\log^{1+\gamma}(-t) +  (1+\gamma)(-t)\log^\gamma(-t) \cdot (-t)^{-1} (-1)= - \log^{1+\gamma}(-t) - (1+\gamma) \log^\gamma(-t) <0,$$
and 
\begin{align*}
p''(t) &=- (1+\gamma) \log^\gamma(-t) (-t)^{-1} (-1) -(1+\gamma) \gamma \log^{\gamma-1}(-t) \cdot (-t)^{-1} (-1)\\
&= \frac{(1+\gamma)\log^\gamma(-t)}{-t}+ \frac{\gamma(1+\gamma)}{(-t) \log^{1-\gamma}(-t)} \cdot
\end{align*}
Hence $p$ is convex decreasing function on $\{t< -10\}$. It follows that $$p(g_x)+ p(g_y) \ge 2 p(\frac{g_x+ g_y}{2}) \gtrsim p(g_x+ g_y)= \tilde{p}.$$
Thus
$$I_1 \le \int_U \rho \tilde{p} \ddc \psi \le \int_U \rho p(g_x) \ddc \psi+ \int_U \rho p(g_y) \ddc \psi.$$
We estimate each term in the right-hand side of the last inequality. Since they are similar, it suffices to treat $ I'_1:=\int_U \rho p(g_x) \ddc \psi$. One has (recall that $\psi$ is smooth)
$$I'_1 =\int_U \rho p(g_x) \ddc (\psi- \psi(x))=  \int_U \rho (\psi- \psi(x)) \ddc (p \circ g_x)+ I_3,$$
where $I_3$ is a sum of integrals whose integrands are forms containing derivatives of $\rho$. It follows that these integrands are zeros on $U_1$. This together with the fact that $p \circ g_x$ is smooth outside $K$ implies that $I_3$ is bounded by $\lesssim \|\psi\|_{L^\infty} \lesssim M$ (uniformly in $\epsilon, u$). Let $\tilde{g}_{x,k}:= \max\{ g_x, -k\}$ for $k \in \N$. We have that $\tilde{g}_{x,k}$ decreases to $\tilde{g}_x$.   Observe 
\begin{align*}
\ddc (p \circ g_x) &= \lim_{k \to \infty} \ddc (p \circ \tilde{g}_{x,k})\\
& = \lim_{k \to \infty} \bigg(  p'(\tilde{g}_{x,k})\ddc\tilde{g}_{x,k} + p''(g_x)d \tilde{g}_{x,k} \wedge \dc \tilde{g}_{x,k} \bigg).
\end{align*}
Hence we see that (remember $x=0$)
\begin{align*}
\int_U \rho (\psi- \psi(x)) \ddc (p \circ g_x) & \lesssim \int_U \rho |\psi(z)- \psi(0)| \frac{\log^\gamma(- \log |z|)}{- \log |z|} \cdot \frac{\omega_\C}{|z|^2}\\
& \lesssim \int_U \frac{\rho  \omega_\C}{\log^{1+\gamma_0-\gamma}(-\log |z|)) (- \log |z|)|z|^2} \\
\lesssim \int_0^{C''} \frac{dr}{r |\log r| \log^{1+ \gamma_0 - \gamma}(-\log r)}<\infty
\end{align*}
provided that $\gamma< \gamma_0$. We infer that 
$$I_1 \lesssim 1$$
if $\gamma < \gamma_0$. This coupled with (\ref{ine-danhgiaI200}) yields that 
$$I \lesssim \frac{1}{ \max\big\{\log^{\gamma/2}\big|\log |x-y|\big|,1 \big\}}$$
for $\gamma< \gamma_0$. The desired assertion hence follows when $\psi$ is $\log^{1+\gamma_0} \log$-continuous. 

It remains to treat the case where $\psi$ is H\"older continuous. We argue similarly. We will choose a different function $p$. As above we obtain $I_1,I_2$ which are to be estimated.  Let $\gamma \in (0, \gamma_1)$ and  $p(t):= e^{-\gamma t}$ for $t <-10$ and $\tilde{p}:=p(g_x+ g_y)$. We have $p'(t)= - \gamma e^{-\gamma t}<0$ and $p''(t)= \gamma^2 e^{-\gamma t}$. As in the first part of the proof, one gets
$$
\int_{A_1} \rho \tilde{p}^{-1}\d h_\epsilon \wedge \dc h_\epsilon \lesssim |x-y|,
$$
and
$$
\int_{A_2} \tilde{p}^{-1}(z)d h_\epsilon \wedge \dc h_\epsilon(z) \lesssim  \int_{0}^\delta \frac{dr}{r^{1-\delta}}= \delta^\gamma/\gamma.
$$
and
$$
\int_{A_3} \tilde{p}^{-1}(z)d h_\epsilon \wedge \dc h_\epsilon(z) \lesssim 
\delta^2 \int_{A_3} \frac{ \omega_{\C}}{|z|^{4-\delta}}\lesssim 
 \delta^2 \int_{ 2 \delta  \le r \le 1} \frac{dr}{r^{3-\delta}} \lesssim \delta^\gamma. 
$$
It follows that $I_2 \lesssim |x-y|^\gamma$. We estimate $I_1$ similarly as before to obtain that $I_1 \lesssim  \int_0^1 r^{-1-\gamma+ \gamma_1} dr <\infty$ because $\gamma < \gamma_1$. Hence we get 
$$|I| \lesssim |x-y|^{\gamma/2}$$
and the desired assertion for H\"older continuity follows.   
\end{proof}

\begin{proof}[Proof of Corollary \ref{cor-family}]
Observe that $\pi^{-1}(K)$ is compact. Since $\pi$ is a submersion, we can cover this set by a finitely many local charts $\mathcal{U}$ in $\mathcal{X}$ such that $\mathcal{U}= Y_1 \times U$, where $Y_1$ is a local chart in $Y$ and $U$ is an open subset in $\C^{m}$, where $m$ is the dimension of fibers of $\pi$. We can indeed assume that $\overline U$ is contained in a bigger local chart of similar forms. 

We only prove the case where $\psi$ is $\log^{1+\gamma_0} \log$-continuous. The H\"older case is done similarly. 
Fix $y \in Y_1$ and let $\psi_y:= \psi|_{\{y\} \times U}$ and $\omega_y:= \omega|_{\{y\} \times U}$. We have  $T_y:= T|_{\{y\} \times U}= \omega_y + \ddc \psi_y$. Since $\omega_y$ is smooth, there exists a constant $C>0$ independent of $y$ such that $\omega_y \le C \omega_{\C^m}$, and $\|\psi_y\|_{\log^{1+\gamma_0} \log(U)} \le C$. Let $T'_y:= C \omega_{\C^n}+ \ddc \psi_y$ which is a closed positive current dominating $T_y$.  Fix $x_0 \in \{y\} \times U$ and let $u(x):= d_{T,y}(x_0,x)$ for $x \in U \backslash V$. As explained in the paragraph before Corollary \ref{cor-family}, we know that $u \in W^*(U)$ and $du \wedge \dc u \le T'_y$. This combined with Theorem \ref{th-bichanLinfinity} yields the desired assertion. 
\end{proof}

In the statement of Theorem \ref{th-bichanLinfinity}, if $\psi$ is merely bounded, then $u$ is not necessarily bounded. We thank Gabriel Vigny for pointing out the following example. 

\begin{example}  Let $u(z):= -\log\big(-\log (-\log |z|^2) \big)$ on the disk $U:=\D_{1/10}$ of radius $1/10$ centered at $0$ in $\C$. We compute
$$\mu:= i \partial u \wedge \bar \partial u = \frac{\omega_\C}{|z|^2 \log^2  |z|^2 \log^2 (-\log |z|^2) } \cdot$$
Since $\mu$ is of finite mass on $U$, we see that $u \in W^*(U)= W^{1,2}(U)$ and $u$ is unbounded.
Recall that 
$$\psi(z):= \int_U \log |z- w| d\mu(w)$$
is a potential of $\mu$, i.e, $\ddc \psi =\mu$. Observe that $\psi$ is smooth outside $0$. We check that $\psi$ is bounded on $U$. Compute
$$\psi(0)= \int_U \log |z|\frac{\omega_\C}{|z|^2 \log^2 |z|^2 \log^2 (-\log |z|^2) }$$
which is
$$\lesssim \int_0^{1/10} \frac{dr}{r \log r^2 \log^2(-\log r^2)}= \int_{-\infty}^{-\log 10} \frac{d t}{ 2t \log^2 (-2t)}< \infty$$
by the change of varibles $t= \log r$.  Hence $\psi$ is bounded on $U$.
\end{example}

\bibliography{biblio_family_MA,biblio_Viet_papers,bib-kahlerRicci-flow}

\begin{thebibliography}{10}

\bibitem{DLW}
{\sc T.-C. Dinh, L.~Kaufmann, and H.~Wu}, {\em Dynamics of holomorphic
  correspondences on {R}iemann surfaces}, Internat. J. Math., 31 (2020),
  pp.~2050036, 21.

\bibitem{DLW2}
\leavevmode\vrule height 2pt depth -1.6pt width 23pt, {\em Random walks on
  {${\rm SL}_2(\Bbb C)$}: spectral gap and limit theorems}, Probab. Theory
  Related Fields, 186 (2023), pp.~877--955.

\bibitem{DKC_Holder-Sobolev}
{\sc T.-C. Dinh, S.~a. Ko{\l}odziej, and N.~C. Nguyen}, {\em The complex
  {S}obolev space and {H}\"{o}lder continuous solutions to {M}onge-{A}mp\`ere
  equations}, Bull. Lond. Math. Soc., 54 (2022), pp.~772--790.

\bibitem{DinhMarinescuVu}
{\sc T.-C. Dinh, G.~Marinescu, and D.-V. Vu}, {\em Moser-{T}rudinger
  inequalities and complex {M}onge-{A}mp\`ere equation}, Ann. Sc. Norm. Super.
  Pisa Cl. Sci. (5), 24 (2023), pp.~927--954.

\bibitem{DS_decay}
{\sc T.-C. Dinh and N.~Sibony}, {\em Decay of correlations and the central
  limit theorem for meromorphic maps}, Comm. Pure Appl. Math., 59 (2006),
  pp.~754--768.

\bibitem{GGZ-logcontiu}
{\sc V.~Guedj, H.~Guenancia, and A.~Zeriahi}, {\em Diameter of {K}\"ahler
  currents}.
\newblock \url{arXiv:2310.20482}, 2023.

\bibitem{Guo-Phong-Song-Sturm}
{\sc B.~G. Guo, D.~H. Phong, J.~Song, and J.~Sturm}, {\em Diameter estimates in
  {K}\"ahler geometry}.
\newblock Comm. Pure Appl. Math., 2022.
\newblock https://doi.org/10.1002/cpa.22196.

\bibitem{Guo-Phong-Song-Sturm2}
\leavevmode\vrule height 2pt depth -1.6pt width 23pt, {\em Sobolev inequalities
  on {K}\"ahler spaces}.
\newblock \url{arXiv:2311.00221}, 2023.

\bibitem{Guo-Phong-Tong-Wang}
{\sc B.~G. Guo, D.~H. Phong, F.~Tong, and C.~Wang}, {\em On the modulus of
  continuity of solutions to complex {M}onge-{A}mp\`ere equations}.
\newblock \url{arXiv:2112.02354}, 2021.

\bibitem{YangLi}
{\sc Y.~Li}, {\em On collapsing {C}alabi-{Y}au fibrations}, J. Differential
  Geom., 117 (2021), pp.~451--483.

\bibitem{Vigny}
{\sc G.~Vigny}, {\em Dirichlet-like space and capacity in complex analysis in
  several variables}, J. Funct. Anal., 252 (2007), pp.~247--277.

\bibitem{Vigny_expo-decay-birational}
{\sc G.~Vigny}, {\em Exponential decay of correlations for generic regular
  birational maps of {$\Bbb{P}^k$}}, Math. Ann., 362 (2015), pp.~1033--1054.

\bibitem{Vigny-Vu-Lebesgue}
{\sc G.~Vigny and D.-V. Vu}, {\em Lebesgue points of functions in the complex
  {S}obolev space}.
\newblock International Journal of Mathematics, 2023.
\newblock https://doi.org/10.1142/S0129167X24500149.

\bibitem{Vu_nonkahler_topo_degree}
{\sc D.-V. Vu}, {\em Equilibrium measures of meromorphic self-maps on
  non-{K}\"{a}hler manifolds}, Trans. Amer. Math. Soc., 373 (2020),
  pp.~2229--2250.

\end{thebibliography}
\bibliographystyle{siam}

\bigskip

\noindent
\Addresses
\end{document}